\documentclass[reqno,12pt]{amsart}		

\usepackage[english]{babel}
\usepackage{amsfonts, amsmath, amssymb, amsthm}

\usepackage{hyperref}
\usepackage{mathtools}  
\mathtoolsset{showonlyrefs}	

\calclayout

\usepackage{amsthm}
\newtheorem{definition}{Definition}
\newtheorem{lemma}[definition]{Lemma}
\newtheorem{proposition}[definition]{Proposition}
\newtheorem*{theorem}{Theorem}
\theoremstyle{remark}

\newtheorem{remark}[definition]{Remark}

\def\Re{\operatorname{Re}}
\def\Im{\operatorname{Im}}

\def\C{\mathbb{C}}
\def\R{\mathbb{R}}
\def\Z{\mathbb{Z}}
\def\N{\mathbb{N}}
\renewcommand{\S}{\mathbb{S}}

\def\D{\mathcal{D}}
\def\supp{\mathrm{supp}}
\def\WF{\mathrm{WF}}

\def\Lie{\mathcal{L}}
\DeclareMathOperator{\tr}{tr}

\newcommand{\vol}{\mathrm{vol}}
\renewcommand{\P}{\mathcal{P}} 

\DeclareMathOperator{\Res}{Res}

\newcommand{\SigmaE}{\Sigma_{\mathrm{e}}}
\newcommand{\SigmaCH}{\Sigma_{\mathrm{ch}}}

\newcommand{\ME}{M_{\mathrm{e}}}
\newcommand{\MCH}{M_{\mathrm{ch}}}

\usepackage{accents}
\DeclareMathAccent{\wtilde}{\mathord}{largesymbols}{"65}
\def\uTilde{\tilde u}
\def\fTilde{\tilde f}

\title[Zeta Function at Zero]
{Zeta Function at Zero for Surfaces with Boundary}
\author{Charles Hadfield}
\address{Department of Mathematics, University of California, Berkeley, CA 94720, USA}
\email{hadfield@math.berkeley.edu}
\begin{document}
\begin{abstract}
We study the Ruelle zeta function at zero for negatively curved oriented surfaces with boundary. At zero, the zeta function has a zero and its multiplicity is shown to be determined by the Euler characteristic of the surface. This is shown by considering certain Ruelle resonances and identifying their multiplicity with dimensions of the relative cohomology of the surface.
\end{abstract}
\maketitle
\thispagestyle{empty}

\section{Introduction}

Consider a compact smooth Riemannian surface $(\Sigma,g)$ without boundary and everywhere strictly negative curvature. The Ruelle zeta function \cite{Ruelle1976Zeta-functionsFlows} provides a differential geometric analogy to the Riemann zeta function by replacing a count over primes with a count over primitive closed geodesics $\{\gamma^\#\}$ whose respective lengths are $\{T_\gamma^\#\}$:
\begin{align}
	\zeta_R(\lambda) := \prod_{\gamma^\#} \left( 1 - e^{-\lambda T_\gamma^\#} \right).
\end{align}
Negative curvature implies this product converges for $\Re\lambda\gg1$. 

This zeta function is related to the Selberg zeta function \cite{Selberg1956HarmonicSeries,Smale1967DifferentiableSystems}
\begin{align*}
	\zeta_R(\lambda)=\frac{\zeta_S(\lambda)}{\zeta_S(\lambda+1)},
	\qquad
	\zeta_S(\lambda) := \prod_{\gamma^\#} \prod_{k \in \N_0} \left(  1 - e^{-(\lambda+k) T_\gamma^\#} \right)
\end{align*}
so results may be translated between the two settings. We will consider only the Ruelle zeta function.

The meromorphic extension of this zeta function has long been known in the setting of constant curvature thanks to the relationship with the Selberg zeta function \cite{Fried1986FuchsianTorsion}. Only recently however has the meromorphic extension been obtained in the setting of variable curvature. This result first appeared in \cite{Giulietti2013AnosovFunctions} and soon after, using a microlocal approach, in \cite{Dyatlov2016DynamicalAnalysis}. In the constant curvature setting, the zeta function vanishes at $\lambda=0$ and its order of vanishing is $-\chi(\Sigma)$ where $\chi(\Sigma)$ is the Euler characteristic of the surface \cite{Fried1986FuchsianTorsion}. This result holds true in variable curvature indicating the topological invariance of the order of vanishing of the zeta function at the origin \cite{Dyatlov2017RuelleSurfaces}. Unlike the constant curvature setting \cite{Fried1986FuchsianTorsion}, the value of the first non-trivial term in the power series representation of the zeta function about the origin is not understood in the variable curvature setting.

Consider now a compact Riemannian surface $(\Sigma,g)$ with strictly convex boundary $\partial \Sigma$ and everywhere strictly negative curvature. In this open setting, one defines the Ruelle zeta function exactly as in the closed case. Strict convexity of the boundary is geometrically appealing as it ensures that closed geodesics do not touch the boundary. Again, negative curvature implies the convergence of the product for $\Re\lambda\gg1$.

For constant curvature, the meromorphic extension of the zeta function has also been understood via the Selberg zeta function and its order of vanishing at $\lambda=0$ is $1-\chi(\Sigma)$ \cite{Patterson2001TheGroups,Borthwick2005SelbergsSurfaces}. See also \cite{Borthwick2016SpectralSurfaces,Guillarmou2018ClassicalSurfaces}. (The one exception to this is if the surface has vanishing Euler characteristic. In this case, the surface is a hyperboloid and the zeta function is a finite product consisting of the two primitive geodesics -- of equal length but opposite direction -- hence the zeta function has a zero of order 2 at the origin.) 

For variable curvature, the zeta function has a meromorphic extension due to \cite{Dyatlov2016Pollicott-RuelleSystems} which considerably extends the microlocal analysis performed in \cite{Dyatlov2016DynamicalAnalysis} by analysing dynamics at both spatial and frequency infinities. This result allows us to consider the zeta function near the origin. Here, we show that the result concerning the order of vanishing discovered in the constant curvature setting holds true in variable curvature.
\begin{theorem}
Let $(\Sigma, g)$ be an oriented Riemannian surface of negative curvature with strictly convex boundary and negative Euler characteristic $\chi(\Sigma)$. Then the Ruelle zeta function $\zeta_R(\lambda)$ has a zero at $\lambda=0$ of multiplicity precisely $1-\chi(\Sigma)$.
\end{theorem}
As with the closed setting, the attractive problem of studying the precise value of the first non-trivial term in the power series representation remains untouched.

We conclude this introduction explaining the method. We also comment on the closed setting for context.

Consider $\Sigma$ a negatively curved compact surface with strictly convex boundary. Let $M=S^*\Sigma$ be the unit cotangent bundle, $\varphi_t:M\to M$ the geodesic flow, and $X\in C^\infty(M;TM)$ the generator of said flow. Let $\Lambda^k T^*_0M$ be the kernel of $\iota_X$ inside $\Lambda^kT^*M$.
The results of \cite{Dyatlov2016Pollicott-RuelleSystems} imply that one may construct the resolvents
\begin{align*}
	(\Lie_X + \lambda)^{-1} : L^2(M; \Lambda^k T^*_0M) \to L^2(M; \Lambda^k T^*_0M)
\end{align*}
which are well-defined for $\Re\lambda\gg1$ and which extend meromorphically to $\lambda\in\C$ (upon a delicate change in the domain and range of said operators). For fixed $k\in\{0,1,2\}$ and a pole $\lambda$ of the meromorphic extension of the resolvent, the associated residue is a projection operator of finite rank whose image defines (generalised) resonant states. These are distributions (or currents for $k>0$) satisfying certain wave-front conditions, support conditions, and are in the kernel of some power of $\Lie_X+\lambda$. Simultaneously the rank of the projection operator is precisely the order of vanishing at $\lambda$ for a certain zeta function $\zeta_k$ associated with $\Lie_X$ acting on $\Lambda^kT^*_0M$. Of course, the relevance of this result is via a factorisation of $\zeta_R$ giving
\begin{align*}
	\zeta_R(\lambda)
    =
    \frac{ \zeta_1(\lambda) }{ \zeta_0(\lambda) \zeta_2(\lambda) }
\end{align*}
hence one can study the order of vanishing of $\zeta_R$ by studying the space of generalised resonant states. Before proceeding, we remark that in all cases of interest, the poles at $\lambda=0$ are simple hence all generalised resonant states are in the kernel of $\Lie_X$ (rather than a power thereof); a result known in \cite{Dyatlov2017RuelleSurfaces} as semisimplicity to which we will return shortly. Due to semisimplicity, we drop the adjective generalised. Denote by $m_k(0)$ the multiplicity of the zero of $\zeta_k$ at $\lambda=0$.

\begin{remark}
Let us briefly comment on the closed manifold setting of \cite{Dyatlov2017RuelleSurfaces}. Resonant states at $\lambda=0$ are
\begin{align*}
	\{ u \in \D'(M;\Lambda^k_0 T^*M) : \WF(u)\subset E_u^*, \Lie_X u =0 \}.
\end{align*}
Here, $E_u^*$ is the unstable subbundle of $T^*M$ associated with the Anosov flow $X$. For $k=0$, the only possible resonant states are constant functions. Moreover, if $\alpha$ denotes the contact form associated with $X$ then an algebraic argument using $d\alpha$, which is parallel with respect to $\Lie_X$, immediately implies $m_2(0)=m_0(0)$. Hence $m_0(0)=b_0$ and $m_2(0)=b_2$, where $b_k$ are the Betti numbers of $\Sigma$. A slightly more difficult task is identifying $m_1(0)$ with $\dim H^1(M)$ (which by Gauss-Bonnet is equal to $b_1$). Up to a semisimplicity argument, the result follows.
\end{remark}

Returning to the present setting of an open manifold. Resonant states at $\lambda=0$ are
\begin{align*}
	\{ u \in \D'(M;\Lambda^k_0 T^*M) : \supp(u)\subset\Gamma_+, \WF(u)\subset E_+^*, \Lie_X u =0 \}.
\end{align*}
Here, $\Gamma_+$ is the set of points trapped in $M$ in backward time with respect to $\varphi_t$, and $E_+^*$ is an extension of the unstable bundle $E_u^*$ (which is only defined on the trapped set) from the trapped set to $\Gamma_+$. 
Due to negative curvature, the volume $V(t)$ of points in $M$ which remain in $M$ after application of $\varphi_t$ decreases exponentially with respect to time. For $k=0$ this implies $(\Lie_X+\lambda)^{-1}$ does not have a pole at $\lambda=0$. Hence $m_0(0)=0$. The same algebraic argument from the closed setting using $d\alpha$ then implies $m_2(0)=0$. It remains to study the space of resonant states for $k=1$. This is done by considering relative cohomology and building an isomorphism between the space of resonant states and $H^1(M,\partial M)$. A key analytic construction allowing this identification is Lemma~\ref{lem:analysis} providing a step between resonant states which are currents and smooth differential forms. (Gauss-Bonnet and Lefschetz duality then imply $m_1(0)=1-\chi(\Sigma)$.) 

The final step is showing simplicity of the pole at $\lambda=0$ for $k=1$. This requires a regularity result very much in the spirit of \cite[Lemma 2.3]{Dyatlov2017RuelleSurfaces} however the argument requires a subtle adaption using ideas from \cite{Dyatlov2016Pollicott-RuelleSystems}.

\subsection*{Acknowledgements} I would like to thank S. Dyatlov and M. Zworski for the many fruitful discussions surrounding this project including many patient explanations of the finer points in their relevant previous work. 

\section{Notation}\label{sec:notation}

\subsection{Geometry}

Let $(\Sigma,g)$ be an oriented Riemannian surface of negative curvature with boundary $\partial \Sigma$. We will also denote by $g$, the genus of $\Sigma$, and by $n$, the number of connected components of the boundary. The Euler characteristic of $\Sigma$ is $\chi(\Sigma)=2-2g-n$ which we take to be negative. Denote by $M$ the unit cotangent bundle of $\Sigma$:
\begin{align}
	M
    :=
    S^*\Sigma 
    =
    \{ (y,\eta) \in T^*\Sigma : g(\eta,\eta)=1 \}.
\end{align}
Let $\alpha \in \Omega^1(M)$ be the pull-back of the canonical one-form on $T^*M$. Then $\alpha$ is a completely non-integrable contact form and we set
\begin{align}
	d \vol_M := \alpha\wedge d\alpha.
\end{align}
The associated Reeb vector field $X\in C^\infty(M;TM)$, which is uniquely determined by
\begin{align}
	\iota_X \alpha = 1,
    \qquad
    \iota_X (d\alpha) = 0,
\end{align}
is the generator of the geodesic flow $\varphi_t : M \to M$. 

Let $\SigmaCH$ denote the convex hull of $\Sigma$. That is, the boundary of $\SigmaCH$ is totally geodesic. Set $\MCH:=S^*\SigmaCH$.

We construct a global frame for $T^*M$ \cite{Singer1967LectureGeometry,Guillemin1980Some2-manifolds}. Denote by $V\in C^\infty(M;TM)$ the generator of the $\S^1$ fibres of $M$ over $\Sigma$. If we denote by $(i\pi):M\to M$ the map given by anticlockwise rotation by $\pi/2$ in the $\S^1$ fibres, then define
\begin{align}
	\beta := (i\pi)^* \alpha \in \Omega^1(M).
\end{align}
Complete the frame by denoting the connection one-form $\omega\in \Omega^1(M)$. This is the unique one-form satisfying
\begin{align}
	\iota_V \omega = 1,
    \qquad
    d\alpha = \omega \wedge \beta,
    \qquad
    d\beta = \alpha\wedge\omega.
\end{align}
Note that $d\vol_M=\alpha\wedge d\alpha=-\alpha\wedge\beta\wedge\omega$, that $d\omega=K\alpha\wedge\beta$ where $K$ is the Gaussian curvature of the surface, and that $\alpha\wedge\beta$ is the pull-back of the area form $d\vol_\Sigma$ determined by the metric $g$.

\subsection{Topology}
We use relative cohomology \`a la Bott and Tu \cite{Bott1982DifferentialTopology}. The vector spaces are $\Omega^k(M)\oplus\Omega^{k-1}(\partial M)$ with differential
\begin{align}
d ( v^{(k)} , h^{(k-1)} ) := ( d v^{(k)} , j^* v^{(k)} - d h^{(k-1)} )
\end{align}
where $j:\partial M\to M$ is inclusion. The cohomology spaces are denoted $H^k(M,\partial M)$.

The first homology group of $\Sigma$ is of rank $2g+n-1=1-\chi(\Sigma)$. Lefschetz duality then implies $H^1(\Sigma,\partial\Sigma)$ is also of rank $1-\chi(\Sigma)$. The Gauss-Bonnet theorem provides
\begin{lemma}\label{lem:cohomology}
Let $\Sigma$ be a surface with boundary whose Euler characteristic is negative. Then $H^1(M,\partial M)$ has rank $1-\chi(\Sigma)$ where $M:=(T^*\Sigma\backslash 0)/\R^+$. 
\end{lemma}
\begin{proof}
We may suppose the surface has a metric whose boundary is totally geodesic, thus we prove the proposition using $(\SigmaCH,g)$ and $\MCH$. We denote by $\pi:\MCH\to\SigmaCH$ the projection and show that $\pi^*:H^1(\SigmaCH,\partial\SigmaCH)\to H^1(\MCH,\partial \MCH)$ is an isomorphism. Let $j$ denote both inclusions $\partial\SigmaCH\hookrightarrow\SigmaCH$ and $\partial \MCH \hookrightarrow \MCH$. As $\partial \SigmaCH$ is totally geodesic, the Gauss-Bonnet theorem reads simply
\begin{align}
	\int_{\SigmaCH} K d\vol_\Sigma = 2\pi \chi(\Sigma).
\end{align}
The proposition follows from surjectivity of $\pi^*$. Let $[(v,h)]\in H^1(\MCH , \partial\MCH)$ and search a candidate $[(w,k)]\in H^1(\SigmaCH,\partial\SigmaCH)$. It suffices to find $f\in\Omega^0(M)$ such that $\iota_Vv=-\Lie_Vf$. (This condition and $dv=0$ imply that $v+df\in\pi^*\Omega^1(\SigmaCH)$, from which we obtain $v=\pi^*w-df$. Similarly, as this condition implies $\Lie_V(h+j^*f)=0$, we may define $k$ by $h=\pi^*k-j^*f$. Therefore $(v-\pi^*w,h-\pi^*k)$ is given by the Bott and Tu differential of $-f$.) Such an $f$ may be constructed if $v$ integrates to zero over the $\S^1$ fibres. We denote this integration as $\pi_*v$ and remark that $\pi_*v$ is constant by Stokes' theorem since all fibres are homotopic. Lifting the Gauss-Bonnet formula to $\MCH$ gives
\begin{align}
	2 \pi \chi(\Sigma)\cdot \pi_*v 
	&= 
    \int_{\MCH} Kv \wedge \alpha\wedge\beta
    =
    \int_{\MCH} -v \wedge d\omega 
    =
    \int_{\partial \MCH} v\wedge \omega
    =
    \int_{\partial \MCH} d h\wedge \omega.
\end{align}
To complete the calculation, we take local coordinates for one component of $\partial \MCH$. Near such a component, the manifold appears as $[0,1]_\rho\times\partial\MCH\simeq [0,1]_\rho\times\partial\Sigma\times\S^1\simeq[0,1]_\rho\times\S^1_t\times\S^1_\theta$. And as $\partial\SigmaCH$ is a geodesic boundary, $j^*\omega=d\theta$. Therefore $dh\wedge\omega$ is the total derivative $d(h\wedge\omega)$ hence vanishes upon integration over $\partial\MCH$. As the Euler characteristic does not vanish, we conclude $\pi_*v=0$ as required.
\end{proof}

\subsection{Dynamics}

Let $\rho$ be a boundary defining function on $M$ (that is, $\rho\in C^\infty(M)$ such that $\rho>0$ on $M^o$, $\rho=0$ on $\partial M$, and $d\rho\neq0$ on $\partial M$). We suppose that $\partial M$ is strictly convex with respect to $X$. That is, we have have the  implication
\begin{align}
	x\in \partial M, (X\rho)(x)=0
    \quad
    \implies
    \quad
    (X^2\rho)(x)<0,
\end{align}
(which is independent of the chosen boundary defining function).
The boundary $\partial M$ decomposes into incoming/tangent/outgoing directions:
\begin{align}
	\partial M = \partial_-M \cup \partial_0M \cup \partial_+M
\end{align}
where
\begin{align}
	\partial_\pm M := \{ x \in \partial M : \pm d\rho(X_x) < 0 \},
    \qquad
    \partial_0 M := \{ x \in \partial M : d\rho(X_x) = 0 \}.
\end{align}
Define the outgoing/incoming tails $\Gamma_\pm \subset M$ and the trapped set $K$ by
\begin{align}
	\Gamma_\pm := \bigcap_{\pm t\ge0} \varphi_t(M),
    \qquad
    K:=\Gamma_+\cap\Gamma_-.
\end{align}
The flow is hyperbolic on $K$. That is, there exists a continuous splitting with respect to $x\in K$ of the cotangent bundle into neutral/stable/unstable bundles each of rank 1 and which is invariant under the flow:
\begin{align}
	T^*_xM = E^*_n(x) \oplus E^*_s(x) \oplus E^*_u(x),
    \qquad
    E^*_n(x)=\R\alpha.
\end{align}
Given a scalar product on $T^*M$, there are constants $C_1,C_2>0$ such that
\begin{align}
	| {\varphi_{-t}}^* \xi |
    \le
    C_1 e^{-C_2 |t|} | \xi |,
    \qquad
    \begin{cases}
    	\xi\in E_s^* & t\ge0;
    	\\
    	\xi\in E_u^*& t\le0.
    \end{cases}
\end{align}
The bundles $E^*_s, E^*_u$ may be extended to $\Gamma_-, \Gamma_+$, respectively. Specifically, there exist subbundles of rank 1, $E^*_\pm\subset T^*_{\Gamma_\pm}M$, which are in the annihilator of $X$, invariant under the flow, depend continuously on $x\in\Gamma_\pm$, and $E^*_+ |_K = E^*_u$ and $E^*_-|_K=E^*_s$. Moreover, if $x\in \Gamma_\pm$ and $\xi\in E^*_\pm$, then as $t\to \mp \infty$, 
\begin{align}
	| {\varphi_{-t}}^* \xi |
    \le
    C_1' e^{-C_2' |t|} | \xi |
\end{align}
for constants $C_1',C_2'$ independent of $(x,\xi)$. \cite[Lemma 1.10]{Dyatlov2016Pollicott-RuelleSystems}

Upon restriction to $\partial M$, the tails $\Gamma_\pm$ are contained in $\partial_\pm M$. Using a metric on $M$, giving a distance function $d(\cdot,\cdot)$, define 
\begin{align}
	\Gamma_\pm^\delta
    :=
    \{ x\in M : d(\Gamma_\pm, x) \le \delta \}.
\end{align}
By taking $\delta$ sufficiently small, we may assume that
\begin{align}
	\Gamma_\pm^\delta \cap \partial M \subset \partial_\pm M.
\end{align}

\section{Zeta function and Pollicott-Ruelle resonances}\label{sec:zeta-resonances}

\subsection{Zeta functions}

Let $\{\gamma\}$ denote the set of geodesics in $M$ and let $\{\gamma^\#\}$ denote the set of primitive geodesics. Given a geodesic $\gamma$, denote respectively by $T_\gamma$ and $T_\gamma^\#$ the length of $\gamma$ and the length of the corresponding primitive geodesic. The Ruelle zeta function is denoted
\begin{align}
	\zeta_R(\lambda) := \prod_{\gamma^\#} \left( 1 - e^{-\lambda T_\gamma^\#} \right)
\end{align}
Denote by $T^*_0M$ the subbundle of $TM$ which annihilates $X$. The pullback ${\varphi_t}^*$ respects the splitting $TM=\R\alpha\oplus T^*_0M$. Given a geodesic $\gamma$ of length $T_\gamma$ and a point $x\in\gamma\subset M$, we introduce the linearised Poincar\'e map
\begin{align}
	\P_{\gamma,x} := {\varphi_{-T_\gamma}}^* : (T^*_0M)_x \to (T^*_0M)_x.
\end{align}
As the endomorphism is conjugate to any other $\P_\gamma (x')$ for $x'\in\gamma$, its determinant and trace are independent of $x$ and under such circumstances we will drop the notation of $x$. We have the following (linear algebra) expression:
\begin{align}
	\det \left( I - \P_\gamma \right) = \sum_{k=0}^2 (-1)^k \tr \Lambda^k \P_\gamma.
\end{align}
A standard manipulation using the preceding expression (as well as the Taylor series for $\log(1-x)$ and, more subtly,  the orientability of the stable and unstable bundles) converts the Ruelle zeta function into an alternating product of zeta functions:
\begin{align}
	\zeta_R(\lambda)
    =
    \frac{ \zeta_1(\lambda) }{ \zeta_0(\lambda) \zeta_2(\lambda) }
\end{align}
where
\begin{align}
	\log \zeta_k(\lambda) := - \sum_{\gamma} \frac{ T_\gamma^\# e^{-\lambda T_\gamma} \tr \Lambda^k \P_\gamma }{ T_\gamma \left| \det (I - \P_\gamma) \right| }.
\end{align}

\subsection{Pollicott-Ruelle resonances}

The Lie derivative with respect to $X$ acting on $\Omega^k(M)$ respects the decomposition $T^*M=\R\alpha\oplus T^*_0M$. Restricting to $T^*_0M$, we consider the transfer operator
\begin{align}
	e^{-t\Lie_X} : C^\infty(M; \Lambda^k T^*_0M) \to C^\infty(M; \Lambda^k T^*_0M).
\end{align}
Given 
	$f\in C^\infty(M),
    u\in C^\infty(M;\Lambda^k T^*_0M)$,
we have $\Lie_X(fu)=(\Lie_Xf)u+f(\Lie_X u)$ from which the transfer operator satisfies
\begin{align}
	e^{-t\Lie_X} (fu) = (\varphi_{-t}^*f)(e^{-t\Lie_X}u).
\end{align}
After having fixed a smooth inner product on $T^*_0M$ (not necessarily invariant under the flow), we have
\begin{align}
	e^{-t\Lie_X} : L^2(M; \Lambda^k T^*_0M) \to L^2(M; \Lambda^k T^*_0M).
\end{align}
Due to the existence of $C_0>0$ such that
\begin{align}
	\| e^{-t\Lie_X} \|_{L^2(M; \Lambda^k T^*_0M) \to L^2(M; \Lambda^k T^*_0M)} \le e^{C_0t},
    \qquad
    t\ge0,
\end{align}
we may define the resolvent $(\Lie_X+\lambda)^{-1}$ on $L^2(M; \Lambda^k T^*_0M)$ for $\Re\lambda>C_0$ by the formula
\begin{align}
	(\Lie_X + \lambda)^{-1} := \int_0^\infty e^{-t(\Lie_X+\lambda)} dt.
\end{align}
A principal result of \cite{Dyatlov2016Pollicott-RuelleSystems}, is that the restricted resolvent
\begin{align}
	R_k(\lambda)
    =
    (\Lie_X+\lambda)^{-1}
    :
    C^\infty_0(M;\Lambda^k T^*_0M)
    \to
    \D'(M;\Lambda^k T^*_0M)
\end{align}
has a meromorphic continuation to $\C$ whose poles are of finite rank. The poles of which are called Pollicott-Ruelle resonances.
Moreover, for each $\lambda_0\in \C$, we have the expansion
\begin{align}
	R_k(\lambda)=R_k^H(\lambda) + \sum_{j=1}^{J(\lambda_0)} \frac{(-1)^{j-1}(\Lie_X+\lambda_0)^j\Pi_{\lambda_0} }{(\lambda-\lambda_0)^j}
\end{align}
where $R_k^H$ is holomorphic near $\lambda_0$ and
\begin{align}
	\Pi_{\lambda_0}
    =
    C^\infty_0(M; \Lambda^k T^*_0M) \to \D'(M; \Lambda^k T^*_0M)
\end{align}
is a finite rank projector. 
The range of $\Pi_{\lambda_0}$ defines generalised resonant states. They are characterised as
\begin{align}
	\Res_k(\lambda_0)
    &:=
    \operatorname{Ran} \Pi_{\lambda_0}
    \\
    &\hphantom{:}=
    \{ u\in\D'(M,\Lambda^k T^*_0M) : \supp (u) \subset \Gamma_+, \WF(u)\subset E^*_+, (\Lie_X+\lambda_0)^{J(\lambda_0)}u=0 \}.
\end{align}
A generalised resonant state is called simply a resonant state if it is in $\ker(\Lie_X+\lambda_0)$.
Finally, it is shown that poles of the meromorphic continuation correspond to zeros of the zeta function $\zeta_k$, and that the rank of the projector $\Pi_{\lambda_0}$ equals the multiplicity of the zero, denoted $m_k(\lambda_0)$.

Let $\mathcal{T}(t)\subset M$ be the set of points $x\in M$ such that $\varphi_{-s}(t)\in M$ for all $s\in[0,t]$, and let $V(t):=\mathrm{Vol}(\mathcal{T}(t))$ be the non-escaping mass function. In our setting, the escape rate
\begin{align}
	Q:= \limsup_{t\to\infty} \frac 1t \log V(t)
\end{align}
is strictly negative \cite[Proposition 2.4]{Guillarmou2017LensSets} thanks to the hyperbolicity of the trapped set, and the strict convexity of the boundary. Hence $V(t)$ decays exponentially fast and \cite[Propostion 4.4]{Guillarmou2017LensSets} provides
\begin{proposition}\label{prop:NoPole}
	The resolvent $R_0(\lambda)$ does not have a pole at $\lambda=0$.
\end{proposition}
We observe that $\Res_2(0)=\{0\}$. Indeed, suppose that $u^{(2)}\in\Res_2(0)$ is a resonant state, that is $\Lie_X u^{(2)}=0$. Since $\Lambda^2T^*_0 M=\R d\alpha$, we have $u^{(2)}=:u^{(0)} d\alpha$ for $u^{(0)}\in\D'(M)$ with $\supp ( u^{(0)} )\subset\Gamma_+$ and $\WF(u^{(0)})\subset E^*_+$. Moreover $\Lie_X u^{(0)}=0$ because $\Lie_X d\alpha=0$ hence $u^{(0)}\in\Res_0(0)=\{0\}$. We have proved
\begin{proposition}
	The resolvent $R_2(\lambda)$ does not have a pole at $\lambda=0$.
\end{proposition}

In Sections~\ref{sec:semisimplicity},~\ref{sec:analysis}, we prove 
\begin{proposition}\label{prop:Semisimplicity}
	The resolvent $R_1(\lambda)$ has a simple pole at $\lambda=0$.
\end{proposition}
Assuming Proposition~\ref{prop:Semisimplicity}, we note that
$\Res_1(0)$ consists only of resonant states:
\begin{align}
	\Res_1(0)
    =
    \{ u\in\D'(M,T^*_0M) : \supp (u) \subset \Gamma_+, \WF(u)\subset E^*_+, \Lie_X u = 0 \}.
\end{align}
and the
Theorem follows if we can show
\begin{align}
	\dim \Res_1(0) = \dim H^1(M,\partial M).
\end{align}

\section{Identification of resonances with relative cohomology}\label{sec:construction-cohomology}
\subsection{Construction of map}

We begin with an analytical result to be proved in Section~\ref{sec:analysis}
\begin{lemma}\label{lem:analysis}
Let $u \in \Res_1(0)$. There exists $f\in \D'(M)$ and $v\in \Omega^1(M)$ such that
\begin{align}
	\supp(f)\subset \Gamma_+^\delta,
    \qquad
    \WF(f)\subset E^*_+,
    \qquad
    \Lie_X f \in C^\infty_0(M),
\end{align}
and $v=u-df$ with $v\in\ker d$.
\end{lemma}
Here, $\Omega^\bullet(M)$ are smooth differential forms up to, and including on, the boundary, while $C^\infty_0(M)$ denotes smooth functions whose support is contained in the interior of $M$.

In order to identify a candidate relative cohomology class, consider $u\in \Res_1(0)$, and construct $v,f$ as in Lemma~\ref{lem:analysis}. We seek an $h\in\Omega^0(\partial M)$ such that $[(v,h)]\in H^1(M,\partial M)$. To this end we first prove
\begin{lemma}\label{lem:exact-v-h}
	For $u\in\Res_1(0)$ the constructed $v=u-df\in\Omega^1(M)$ of Lemma~\ref{lem:analysis} is exact upon pull-back to $\partial M$.
\end{lemma}
\begin{proof}
We simplify the exposition by supposing $\partial M$ consists of a single connected component (which is isomorphic to a torus). Noting that $\pi_1(\partial M)=\pi_1(\partial \Sigma)\times \pi_1(\S^1)=\Z^2$, it then suffices to show that $\int_{\gamma_i} v=0$ where $\gamma_1,\gamma_2$ are two simple closed curves which generate $\pi_1(\partial M)$. Take $\gamma_1$ to be the curve which corresponds to the generator of $\pi_1(\partial\Sigma)$, and $\gamma_2$ corresponding to the $\S^1$ fibre.

We now take local coordinates similar to Lemma~\ref{lem:cohomology}. For the moment, we work near $\partial M$ rather than $\partial \MCH$. The manifold appears as $[0,1]\times\partial\Sigma\times\S^1\simeq[0,1]_\rho\times\S^1_t\times\S^1_\theta$. We may choose $\gamma_1$ such that its image is entirely contained in $\partial_-M$. As the restrictions of $u,f$ are contained in $\Gamma_+,\Gamma_+^\delta\subset \partial_+M$ respectively, we obtain immediately
\begin{align}
	\int_{\gamma_1} v = 0.
\end{align}

In order to show $\int_{\gamma_2} v$ also vanishes, we work with $\MCH$. Recall the push-forward map $\pi_* : \Omega^1(M)\to\Omega^1(\Sigma)$ which, as $\pi,j$ commute, provides a push-forward $\pi_* : \Omega^1(\partial \MCH)\to\Omega^1(\partial\SigmaCH)$. As the $\S^1$ fibres are homotopic to eachother, $\int_{\gamma_2} v=\pi_* v$. Gauss-Bonnet over $\SigmaCH$ lifts to $\MCH$ as in Lemma~\ref{lem:cohomology}:
\begin{align}
	2\pi\chi(\Sigma) \cdot \pi_* v
    &=
    \int_{\MCH} -v\wedge d\omega
    =
    \int_{\partial_{\MCH}} v \wedge \omega.
\end{align}
With coordinates $[0,1]_\rho\times\S^1_t\times\S^1_\theta$, the curvature form restricted to $\partial\MCH$ is simply $d\theta$. Therefore, writing $v=v_\rho d\rho+v_t dt + v_\theta d\theta$,
\begin{align}
	2\pi\chi(\Sigma) \cdot \pi_* v
    &=
	\int_{\partial\MCH} v_t dt\wedge d\theta
    =
    \int_0^{2\pi}\left( \int_0^{2\pi} v_t dt \right) d\theta
    =
    \int_0^{2\pi} 0 d\theta
    =0
\end{align}
because $\int_0^{2\pi} v_t dt$ vanishes from the prior calculation showing $\int_{\gamma_1} v=0$. As $\chi(\Sigma)<0$, we conclude
\begin{align*}
	\int_{\gamma_2} v = 0.
\end{align*}
Therefore $[j^*v]=0\in H^1(\partial M)$ implying the existence of the required $h\in\Omega^0(\partial M)$ such that $j^*v=dh$.
\end{proof}

\begin{remark}\label{rem:support-h}
The previous lemma ensures that it is possible to define an $h\in \Omega^0(\partial M)$ such that $[(v,h)]\in H^1(M,\partial M)$. However there are $n-1$ degrees of freedom in the choice of $h$ due to the $n$ connected components of $\partial M$. (An overall constant would not be seen by relative cohomology.) These degrees of freedom are fixed by the following declaration: The form $j^*v$ has support contained in $\Gamma_+^\delta \subset \partial_+M$ therefore $dh=0$ on $\partial_-M$ and is therefore constant on $\partial_-M$. We declare that $h$ must be chosen to vanish on $\partial_-M$ whence we may assume $\supp(h)\subset \partial_+M$.
\end{remark}

\begin{proposition}\label{prop:well-defined}
Lemma~\ref{lem:analysis}, Lemma~\ref{lem:exact-v-h}, and Remark~\ref{rem:support-h} establish a well-defined map
\begin{align}
	\Res_1(0) \ni u \longmapsto [(v,h)]\in H^1(M,\partial M).
\end{align}
\end{proposition}

\begin{proof}
Suppose Lemma~\ref{lem:analysis} provides $f_i$ and $v_i=u-df_i$ for $i\in\{1,2\}$. Then Lemma~\ref{lem:exact-v-h} and Remark~\ref{rem:support-h} provide $h_i$ with $j^* v_i = d h_i$ and $h_i$ vanish on $\partial_-M=0$. We aim to construct $k\in \Omega^0(M)$ such that 
\begin{align*}
(dk,j^*k) = (v_1-v_2, h_1-h_2)
\end{align*}
in order to verify that the relative cohomology class is independent of the choices made.

Set $k:= f_2-f_1$. As $dk=v_1-v_2$ is smooth, we conclude $k$ itself is smooth. Next,
\begin{align*}
d(j^*k) = j^*d(f_2-f_1) = j^*(v_1-v_2) = d(h_1-h_2)
\end{align*}
so $j^*k= h_1-h_2 + c$ where $c$ is a function constant on each connected component of $\partial M$. As $h_i$ vanish on $\partial_-M$ and $\supp(k)\subset \Gamma_+^\delta$, we conclude the function $c$ is the zero function.
\end{proof}

\subsection{Injectivity}

Given the notation established from the previous subsection, suppose that for a given $u\in\Res_1(0)$, we obtain $[(v,h)]=0\in H^1(M,\partial M)$. This implies the existence of $k\in \Omega^0(M)$ whose Bott and Tu differential gives $(v,h)$. That is, $(dk,j^*k)=(v,h)$. This implies that $u=d(f+k)$. However as $u$ vanishes away from $\Gamma_+$, we know that $f+k$ is smooth on $M\backslash \Gamma_+$ and is constant on each connected component of $M\backslash \Gamma_+$. There are $n$ connected components of $M\backslash \Gamma_+$ and each component may be identified with the $n$ connected components of $\partial_-M$ (upon following geodesics in backward time until they reach the boundary). So the value of $f+k$ on a connected component is determined by its value upon restriction to the corresponding component of $\partial_-M$. Now $\supp(f)\subset\Gamma_+^\delta$ and $j^*k=h$ which by Remark~\ref{rem:support-h} vanishes on $\partial_-M$. Therefore $f+k=0$ on $M\backslash\Gamma_+$ and upon observing
\begin{align}
	\supp(f+k)\subset \Gamma_+,
    \qquad
    \WF(f+k) \subset E^*_+,
    \qquad
    \Lie_X (f+k)=\iota_X u =0,
\end{align}
we conclude $f+k\in\Res_0(0)=\{0\}$ hence $u=0$.

\subsection{Surjectivity}\label{subsec:surjectivity}

Consider an element of $H^1(M,\partial M)$. Suppose it takes the form $[(\tilde v, \tilde h)]$. We first remark that $\tilde h$ may be extended to a smooth function on $M$ whose Bott and Tu differential gives $0\in H^1(M,\partial M)$ and which may be subtracted from our original element. We may therefore assume the element of $H^1(M,\partial M)$ takes the form $[(\tilde v,0)]$ for some modified $\tilde v$.

By \cite[Section 4]{Guillarmou2017LensSets}, there exists $\tilde f\in\D'(M)$ with $\WF(\tilde f)\subset E^*_+$ subject to the  boundary value problem
\begin{align}
	\begin{cases}
	\Lie_X \tilde f=-\iota_X \tilde v;
    \\
    \tilde f|_{\partial_-M}=0.
	\end{cases}
\end{align}
Set $u:=\tilde v+d\tilde f$. Immediately, $\iota_X u =0$ and since $\tilde v$ is closed, $\Lie_X u=0$. It remains to obtain a support condition on $u$ to conclude that $u$ is a resonant state. To this end, consider a point $x\in \partial_-M$ and $U$ a neighbourhood in $\partial_-M$ of $x$. Locally near $x$, $X$ is transversal to $\partial_-M$ and is incoming. We may thus consider a chart $[0,\varepsilon)_\rho\times U_{(t,\theta)}$ on which $X$ takes the form $\partial_\rho$. Writing
\begin{align}
	u|_{[0,\varepsilon)\times U} = u_\rho d\rho + u_t dt + u_\theta d\theta
\end{align}
we see that $u_\rho=0$ (since $\iota_X u=0$) and that $u_t,u_\theta$ are independent of $\rho$ (since $du=0$). So $u_t,u_\theta$ are determined by their values on $\{0\}\times U$ but by the initial condition of the boundary value problem
\begin{align}
	j^* u |_{\partial_-M} = j^*(\tilde v+d\tilde f) |_{\partial_-M} = d(0+j^*\tilde f) |_{\partial_-M} = 0.
\end{align}
Therefore $u$ vanishes on a neighbourhood of any point in $\partial_-M$. Moreover $u$ is smooth away from $\Gamma_+$ and in the kernel of $\Lie_X$. Therefore $\supp(u)\subset \Gamma_+$ hence $u\in\Res_1(0)$.

To be completely at peace, we ought check that $u$ gives back the original cohomology element following Proposition~\ref{prop:well-defined}. The argument is that of the proof of Proposition~\ref{prop:well-defined}: Suppose Lemma~\ref{lem:analysis} provides $f$ and $v=u-df$. Then Lemma~\ref{lem:exact-v-h} and Remark~\ref{rem:support-h} provide $h$ with $j^*v=dh$ and $h$ vanishes on $\partial_-M$. We must construct $k\in \Omega^0(M)$ such that 
\begin{align*}
(dk,j^*k) = (v-\tilde v, h).
\end{align*}
Set $k:=\tilde f - f$. As $dk= v-\tilde v$ is smooth, so too is $k$. Also $d(j^*k)=dh$ and both $k$ and $h$ vanish  on $\partial_-M$ so $j^*k=h$.

\section{Semisimplicity}\label{sec:semisimplicity}

Proposition~\ref{prop:Semisimplicity} states that that the pole of $R_1(\lambda)$ at $\lambda=0$ is simple. The proof consists of two parts; a microlocal argument and a geometrical argument. Lemma~\ref{lem:flux-simple} announces the microlocal result in a simplified form. The more general microlocal result is Lemma~\ref{lem:flux-full} stated and proved in Section~\ref{sec:analysis}.

\begin{lemma}\label{lem:flux-simple}
Let $f'\in\D'(M)$ with $\WF(f')\subset E_+^*$. If $\Lie_X f'\in C^\infty_0(M)$ and
\begin{align}
\Re \int_M \Lie_X f' \, \overline{f'} \, d\vol_M =0,
\end{align}
then $f'\in C^\infty(M)$.
\end{lemma}

Consider a generalised resonant state $\uTilde$ associated with the resonance $\lambda=0$. That is, $\uTilde \in \D'(M; T^*_0M)$ with $\supp (\uTilde) \subset \Gamma_+$ and $\WF(\uTilde)\subset E^*_+$. In order to show that all generalised resonant states are true resonant states it suffices to assume $\Lie_X^2 \uTilde =0$ and show that $\Lie_X \uTilde =0$. To this end suppose $u:=\Lie_X \uTilde \in \ker \Lie_X$.

Define $\fTilde$ by $\fTilde d\vol_M:=\alpha\wedge d\uTilde$. Then $(\Lie_X \fTilde) d\vol_M = \alpha\wedge d u =0$ since $du\in\Res_2(0)=\{0\}$. So $\Lie_X \fTilde=0$ hence $\fTilde\in\Res_0(0)=\{0\}$. We conclude that $\alpha\wedge d \uTilde$ vanishes hence $\alpha\wedge u$ is exact:
\begin{align}
	d\uTilde = \alpha \wedge u.
\end{align} 

As $u\in\ker\Lie_X$, we consider a representative of the relative cohomology class obtained from $u$ as constructed in Section~\ref{sec:construction-cohomology}. Following Section~\ref{sec:construction-cohomology}, introduce 
$f\in \D'(M)$, $v\in \Omega^1(M)$, and $h\in C^\infty(\partial M)$ such that
\begin{align}
	\supp(f)\subset \Gamma_+^\delta,
    \qquad
    \WF(f)\subset E^*_+,
    \qquad
    \Lie_X f \in C^\infty_0(M),
\end{align}
and $v=u-df$ with $j^*v = dh$. Note $dv=0$ and $\supp(\iota_X v)\subset M^o$. Moreover $\supp(h)\subset \partial_+M$ by Remark~\ref{rem:support-h}.

We extend $h$ to $h'\in C^\infty(M)$ such that $\Lie_X h'\in C^\infty_0(M)$ and $j^*h'=h$ in the following way.
Let $U\subset\partial M$ be open and relatively compact in $\partial_+M$ such that $\supp(h)\subset U$. As $X$ is transversal and outgoing on $U$, there exists a chart
\begin{align}
	[0,\varepsilon)_\rho \times U_{(t,\theta)} \subset M
\end{align}
on which $X=-\partial_\rho$. Consider a cutoff $\chi\in C^\infty([0,\varepsilon);[0,1])$ with $\supp(\chi)\subset [0,\varepsilon/2]$ and $\chi|_{[0,\varepsilon/3]}=1$. Now define $h'\in C^\infty(M)$ by declaring 
\begin{align}
	h'(\rho,t,\theta):=\chi(\rho) \cdot h(t,\theta)
\end{align}
on $[0,\varepsilon)\times U$ and $h'=0$ elsewhere. Then $h'$ has the desired properties. We remark
$[(dh',h)]=0$ in $H^1(M,\partial M)$ since $(dh',h)$ is the Bott and Tu differential of $h'$.

Motivated by the equality $[(v,h)]= [(v,h)]-[(dh',h)]$, set
\begin{align}
	f':=f+h',
	\qquad
    v':=u-df'=v-dh'.
\end{align}
Then $j^*v'=0$, $dv'=0$, and $\Lie_X f'\in C^\infty_0(M)$. In order to place ourselves in the setting of Lemma~\ref{lem:flux-simple} we calculate
\begin{align}
	\int_M \Lie_X f' \, \overline{f'} \, d\vol_M
		&=
        - \int_M \overline{f'} \, \iota_X v' \, d\vol_M
        \\
        &=
        - \int_M \overline{f'} \, v' \wedge d\alpha
        \\
        &=
        - \int_M \overline{df'} \wedge v' \wedge \alpha
        \\
        &=
        - \int_M \overline{u-v'} \wedge v' \wedge \alpha
\end{align}
Note that the third equality does not produce a boundary term as $j^*v'=0$.
Passing to the real part of the preceding equality, the term involving $\overline{v'}\wedge v$ vanishes. Using the exactness of $\alpha\wedge u$ and (for a second time) that $dv'=0$ we conclude
\begin{align}
	\Re \int_M \Lie_X f' \, \overline{f'} \, d\vol_M
    	&=
        - \Re \int_M \overline{u} \wedge v' \wedge \alpha
		=
        - \Re \int_M v' \wedge \overline{ d \uTilde }
        =
        0.
\end{align}
By Lemma~\ref{lem:flux-simple}, we deduce $f'$ is smooth hence $u$ is smooth. Therefore $u$ is forced to vanish as $\supp(u)\subset \Gamma_+$. This finishes the proof as it shows that the generalised resonant state $\uTilde$ is actually a true resonant state. 

\section{All things analysis}\label{sec:analysis}

This final section tidies up the loose analytical threads and proves Lemmas~\ref{lem:analysis} and~\ref{lem:flux-simple}. Due to the microlocal nature of the proofs, it is easier to work on either a compact manifold without boundary (in the spirit of \cite{Dyatlov2016Pollicott-RuelleSystems}) or on an open manifold (in the spirit of \cite{Guillarmou2017LensSets}) rather than on a compact manifold with boundary. We choose to work with an open manifold. We consider $\Sigma$ inside a small extension $\SigmaE$ which is non-compact. For example, using a boundary defining function $\rho$ giving the collar neighbourhood $[0,1]_\rho\times\partial\Sigma$ of $\partial\Sigma$ in $\Sigma$, we could define $\SigmaE$ to be $\left( (-\varepsilon,0)_\rho\times \partial\Sigma \right) \cup \Sigma$. We may extend $g$ to a negatively curved metric on $\SigmaE$ and construct $\ME:=S^*\SigmaE$ such that the hypersurfaces $\{ \rho = - \varepsilon' \}$ for $\varepsilon'<\varepsilon$ are all strictly convex hypersurfaces inside $\ME$. The constructions of Section~\ref{sec:notation} extend from $M$ to $\ME$. Note also that the spectral projectors obtained in Section~\ref{sec:zeta-resonances} are almost identical to the spectral projectors one obtains on $\ME$. Precisely, if $\Pi_{\mathrm{e},\lambda_0}$ denotes the spectral projector on $\ME$ then restricting its domain to compactly supported sections on $M$ and then restricting the resulting distribution on $\ME$ to $M$ recovers $\Pi_{\lambda_0}$. We will therefore not distinguish the projectors and continue to use the notation from Section~\ref{sec:zeta-resonances}.

\subsection{Proof of Lemma~\ref{lem:analysis}}
We start by collecting three separate ideas of the proof in the following three remarks.
\begin{remark}\label{remark:1}
Naively, following the construction in the compact without boundary setting, introduce a metric on $\ME$, construct the Hodge Laplacian $\Delta$, the divergence $d^*$, and a paramatrix $Q$ for $\Delta$. One sets $f:=d^*Qu$ and then checks $u-df$ is smooth by the following argument. First,
\begin{align*}
	u-df
    &=
    u-d d^* Qu - d^* dQu + d^*dQu
    \\
    &=
    u-\Delta Qu + d^* dQu
    \\
    &= d^*dQu + \Omega^1(\ME)
\end{align*}
By elliptic regularity, it suffices to then show $\Delta d^* d Q u$ is smooth which follows because $\Delta$ commutes with $d^*$ and $d$ and, as $u$ is a resonant state, $du$ is smooth.
\begin{align*}
	\Delta d^* d Q u = d^* d \Delta Q u = d^* d ( u + \Omega^1(\ME) ).
\end{align*}
\end{remark}
\begin{remark}\label{remark:2}
Consider $u\in \D'(\ME;\Lambda^k T^*\ME)$ with $du$ smooth and let $\chi\in C^\infty(\ME;[0,1])$ with $\supp(\chi)\subset \Gamma_+^\delta$ and $\chi=1$ on $\Gamma_+^{\delta/2}$. Then set $f:=\chi d^* Q u$ and the following argument shows that $u-df$ is smooth. First $Qu$ is smooth away from $\Gamma_+$ hence $[d,\chi] d^*Qu$ is smooth since $d\chi$ vanishes near $\Gamma_+$. Then
\begin{align*}
	u-df = u- \chi dd^*Qu - [d,\chi] d^*Qu = \chi ( u - dd^*Qu) + \Omega^k(\ME)
\end{align*}
and $u-dd^*Qu$ is smooth by the previous remark. 
\end{remark}

\begin{remark}\label{remark:3}
Consider a closed Riemannian manifold $(Y,g_Y)$ and the cylinder $\R_\rho\times Y$ with metric $g=d\rho^2+g_Y$ and projection $\pi_Y:\R\times Y\to Y$. Consider further the Hodge Laplacians $\Delta_Y,\Delta$ and let $Q_Y,Q$ be respective paramatrices, finally, let $d_Y^*,d^*$ be the adjoints of the de Rham differentials. Consider a form $u_Y\in \D'(Y;\Lambda^kT^*Y)$ and set $u:=\pi_Y^* u_Y$. After constructing
\begin{align}
	f_Y:=d_Y^*Q_Y u_Y,
    \qquad
    f:=d^*Qu,	
\end{align}
we consider the difference $\pi_Y^* f_Y - f$. This difference is smooth due to the following argument. By elliptic regularity, it suffices to show $\Delta ( \pi_Y^* f_Y - f )$ is smooth. Due to the product structure of $g$, we have $\Delta \pi_Y^* f_Y = \pi_Y^* \Delta_Y f_Y$. Similarly $\pi_Y^* d_Y^* u_Y=d^* u$. Therefore
\begin{align*}
	\Delta \left( \pi_Y^* f_Y - f \right)
    	&=
        \pi_Y^* d_Y^* \Delta_Y Q_Y u_Y - d^* \Delta Q u
        \\
        &=
        \pi_Y^* d_Y^* \left( u_Y + \Omega^k(Y) \right) - d^* \left ( u + \Omega^k(\R\times Y) \right)
        \\
        &=
        d^*u - d^*u + \Omega^k(\R\times Y).
\end{align*}
\end{remark}

\begin{proof}[Proof of Lemma~\ref{lem:analysis}.]
Consider $u\in \D'(\ME;\Lambda^1 T^*_0\ME)$ a resonant state. We can choose an open set $U\subset\partial M$ containing $\Gamma_+\cap\partial M$ such that on the chart $V:=(-\varepsilon,\varepsilon)_\rho \times U \subset \ME$ the geodesic flow takes the form $-\partial_\rho$. As in Subsection~\ref{subsec:surjectivity}, we conclude that on $V$, we may write $u=\pi_\partial^* u_\partial$ for some $u_\partial \in \D'(\partial M;\Lambda^1 T^*\partial M)$ where $\pi_\partial:(-\varepsilon,\varepsilon)\times \partial M\to\partial M$ is projection onto the boundary. Let $\chi\in C^\infty(\ME;[0,1])$ with $\supp(\chi)\subset \Gamma_+^\delta$ and $\chi=1$ on $\Gamma_+^{\delta/2}$. We may assume that
\begin{align*}
	\supp ( \Lie_X \psi ) \cap \{ \rho \le \tfrac{2\varepsilon}3\} = \varnothing,
\end{align*}
and that $\chi=\chi_1+\chi_2$ with
\begin{align*}
	\supp (\chi_1) \subset V\cap\{\rho<\tfrac{2\varepsilon}3\},
	\qquad
    \supp(\chi-\chi_1) \cap\{\rho\le\tfrac\varepsilon3\}=\varnothing.
\end{align*}
Now choose a metric on $\ME$ and ask that on $V$, the metric takes the product structure of Remark~\ref{remark:3}. Construct paramatrices $Q_\partial,Q$ for the Hodge Laplacian acting on $\Omega^1(\partial M),\Omega^1(\ME)$, and denote the divergences on $\partial M, \ME$ by $d_\partial^*,d^*$ respectively. Construct
\begin{align}
	f_\partial := d_\partial^* Q_\partial u,
    \qquad
    \tilde f := d^* Q u,
    \qquad
    f:= \chi_1 \pi_\partial^* f_\partial + \chi_2 \tilde f.
\end{align}
We claim that $f$ satisfies the conclusions of Lemma~\ref{lem:analysis}. By construction of $\chi$, we have the correct support condition for $\Lie_X f$ and so it remains to show that $v:=u-df$ is smooth. First,
\begin{align*}
	v 
    = u-df 
    = \chi_1 \pi_\partial^*( u_\partial - d f_\partial ) + \chi_2 ( u - d\tilde f)
    	-[d,\chi_1] \pi_\partial^* f_\partial
        -[d,\chi_2] \tilde f
\end{align*}
and the first two terms are smooth by Remark~\ref{remark:1}. The second two terms provide
\begin{align*}
	[d,\chi_1] \pi_\partial^* f_\partial
        +[d,\chi_2] \tilde f
    =
    (\partial_\rho \chi_1 ) d\rho \wedge (\pi_\partial^* f_\partial - \tilde f)
\end{align*}
which is smooth due to Remark~\ref{remark:3}.
\end{proof}

\subsection{Proof of Lemma~\ref{lem:flux-simple}}
Finally we will deduce
Lemma~\ref{lem:flux-simple}
from 
Lemma~\ref{lem:flux-full}
whose proof relies on a well-chosen cut-off in the spirit of the escape functions dating to \cite{Faure2011UpperFlows}. The microlocal and semiclassical notation used is standard with short and sufficient introductions given in \cite{Dyatlov2016DynamicalAnalysis,Dyatlov2016Pollicott-RuelleSystems} or more substantially in \cite{Hormander2007TheIII,Zworski2012SemiclassicalAnalysis}.

Geometrically, we have the radially compactified tangent bundle $\bar T^*\ME$ (which is not compact as $\ME$ is an open manifold), the sphere bundle at infinity $S^*\ME$ (which upon introduction of an arbitrary metric on $\ME$ may be thought of as the unit tangent bundle as suggested by the notation), and the projections $\kappa : \bar T^*\ME\to S^*\ME$, $\pi:S^*\ME\to\ME$. Analytically, consider the operator $P:=-i\Lie_X$ which is formally self-adjoint on $L^2(\ME)$ with respect to the volume form $d\vol_{\ME}$. Let $p$ denote its symbol, $p(x,\xi)=\xi(X_x)$ and let $H_p$ denote the associated Hamiltonian vector field on $T^*\ME$ which, upon application of $\kappa_*$, may also be viewed on $S^*\ME$.

\begin{lemma}\label{lem:escape-function}
Let $A$ be a compact subset of $T^*\ME\backslash0$ and let $D$ be a compact subset of $\Gamma_+\subset\ME$. Then there exists $\chi\in C^\infty_0(T^*\ME;[0,1])$ such that 
\begin{itemize}
\item $\chi=1$ near $\{(x,\xi)\in T^*\ME : x\in D,\xi=0\}$;
\item $H_p\chi\le0$ near $\{(x,\xi)\in T^*\ME : x\in \Gamma_+, \xi\in E^*_+ \}$;
\item $H_p\chi<0$ on $E_+^*\cap A$.
\end{itemize}
\end{lemma}
\begin{proof}
Let $\mathcal{U}:=\kappa(E^*_+)\cap ( \kappa(A)\cup \pi^{-1}(D) )$. Let $U_2$ be open in $S^*\ME$, contain $\mathcal{U}$, and be relatively compact in $S^*\ME$. Let $U_1\subset U_2$ be open with $\mathcal{U}\subset \overline U_1\subset U_2$. Define $\chi_1\in C^\infty(T^*\ME)$ with
\begin{align*}
	\chi_1|_{\pi(U_1)} = 1,
    \qquad
	\supp ( \chi_1 ) \subset \pi(U_2).
\end{align*}

Next consider $(x,\xi)\in \kappa^{-1}(\mathcal{U})$. There exists $T>0$ (independent of $(x,\xi)$ due to the compactness of $\mathcal{U}$) such that $| e^{-TH_p}(x,\xi) | \le \tfrac12 | (x,\xi) |$. (Here, $|\cdot|$ is an arbitrary metric on the fibres.) For $U_2$ chosen sufficiently small, this property holds on $U_2$ as well. Define a homogeneous of degree 1 function $\chi_2\in C^\infty(\kappa^{-1}(U_2))$ as
\begin{align*}
	\chi_2 := \int_{-T}^0 |e^{tH_p}| dt,
    \qquad
    (H_p \chi_2)(x,\xi) = |(x,\xi)|-| e^{tH_p}(x,\xi) | \ge \tfrac12 |(x,\xi)|.
\end{align*}

Finally construct $\chi_3\in C_0^\infty (\R ; [0,1])$ such that
\begin{align*}
	\chi_3 = 1 \textrm{ near } 0,
    \qquad
    \chi_3'\le 0 \textrm{ on } [0,\infty),
    \qquad
    \chi_3' < 0 \textrm{ on } \chi_3(E^*_+\cap A).
\end{align*}
(Note that the function $\chi_3\circ\chi_2$ extends smoothly to a function on $T^*_{\pi(U_2)} \ME$.)  The function $\chi:=\chi_1 \cdot (\chi_3\circ\chi_2)$ satisfies the required conditions.
\end{proof}

\begin{lemma}\label{lem:flux-full}
Suppose $u\in \D'(\ME)$ with $\WF(u)\subset E^*_+$ and $Pu\in C^\infty_0(\ME)$. If $\Im \langle Pu,u\rangle \ge 0$ then $u\in C^\infty(\ME)$.
\end{lemma}
\begin{proof}
The hypotheses imply $\WF_\hbar(u) \cap (T^*\ME\backslash 0) \subset E^*_+$ and $\WF_{\hbar}(Pu) \cap (\bar T^*\ME\backslash0)=\varnothing$. It is sufficient to prove:
\begin{itemize}
\item
	Given $A\in\Psi_\hbar^{\mathrm{comp}}(\ME)$ with $\WF_\hbar (A) \subset T^*M \backslash 0$, there exists $B\in \Psi_\hbar^{\mathrm{comp}}(\ME)$ with $\WF_\hbar (B) \subset T^*M \backslash 0$ such that $\| Af \|_{L^2} \le C \hbar^{1/2} \| Bf \|_{L^2} + \mathcal{O}(\hbar^\infty)$.
\end{itemize}
Indeed, induction then gives $\|Af\|_{L^2} = \mathcal{O}(\hbar^\infty)$ implying $\WF_\hbar(u)\cap (T^*\ME \backslash0)=\varnothing$ so $\WF_\hbar(u)=\varnothing$ hence $u\in C^\infty(\ME)$. To this end consider $\hbar P\in\Psi^1_\hbar(\ME)$ with principal symbol $\sigma_\hbar(\hbar P)=p$ where $p(x,\xi)=\xi( X_x)$. By Lemma~\ref{lem:escape-function}, there exists $\chi\in C^\infty_0(T^*M;[0,1])$ such that
\begin{itemize}
\item $\chi=1$ near $\WF_h(P u) = \{ (x,\xi) \in T^*\ME : x\in\supp(Pu), \xi=0 \}$;
\item $H_p\chi\le 0$ near $\WF_\hbar(u) \subset \{ (x,\xi) \in \bar T^*\ME : x\in\Gamma_+, \xi\in E^*_+ \}$;
\item $H_p\chi<0$ on $E^*_+\cap \WF_\hbar(A)$.
\end{itemize}
Quantising $\chi$ gives a self-adjoint operator $F\in \Psi^{\mathrm{comp}}_\hbar(\ME)$ whose principal symbol returns $\chi$ and for which
\begin{align}
	\WF_\hbar(I-F) \cap \supp(Pf) \subset \bar T^*M\backslash 0.
\end{align}

Now $H_p\chi\le 0$ near $E_+^*$ and $H_p\chi<0$ on $E^*_+\cap \WF_\hbar(A)$ so we can construct $A_1\in\Psi^\mathrm{comp}_\hbar(\ME)$ with $\WF_\hbar(A_1)\subset T^*\ME\backslash 0$ and $\WF_\hbar(A_1)\cap E^*_+=\varnothing$ such that
\begin{align}
	-\tfrac12 H_p\chi + |\sigma_h(A_1)|^2 \ge C^{-1} |\sigma_\hbar(A)|^2.
\end{align}
One may now construct $B\in\Psi^\mathrm{comp}_\hbar(\ME)$ such that $\WF_\hbar(B) \subset T^*M\backslash 0$ and
\begin{align}
	\left(
    	\WF_\hbar [P,F] \cup \WF_\hbar(A) \cup \WF_\hbar(A)
    \right)
    \cap
    \WF_\hbar(I-B)
    =
    \varnothing.
\end{align}
As $\chi\equiv 1$ near $\WF_\hbar(Pu)$, we obtain $(I-F)Pu=\mathcal{O}(\hbar^\infty)$. The hypothesis $\Im\langle Pu,u\rangle \ge0$ and the symmetry of $F$ give $\langle \tfrac 1{2i} [P,F]u,u\rangle \le \mathcal{O}(h^\infty)$ so the sharp G{\aa}rding inequality \cite[Lemma 2.4]{Dyatlov2016Pollicott-RuelleSystems} with
\begin{align}
	\tfrac1{2i}[P,F]+A_1^*A_1-C^{-1}A^*A\ge0
\end{align}
and $Bu\in L^2(\ME)$ provides
\begin{align}
	\langle \tfrac1{2i}[P,F]Bu,Bu\rangle + \|A_1Bu\|^2_{L^2} - C^{-1} \| ABu \|^2_{L^2} \ge -\hbar \|Bu\|^2_{L^2}.
\end{align}
Noting that $\|A_1Bu\|^2_{L^2}$ and $\langle \tfrac1{2i}[P,F]Bu,Bu\rangle$ are $\mathcal{O}(\hbar^\infty)$, and that $ABu=Au$ by construction of $B$, we obtain the inequality $\| Au \|_{L^2} \le C \hbar^{1/2} \| Bu \|_{L^2} + \mathcal{O}(\hbar^\infty)$.
\end{proof}

\bibliographystyle{alpha}
\bibliography{Mendeley}
 \end{document}